\documentclass[12pt]{article}
\usepackage{amsmath,amssymb,graphicx,amsthm}
\usepackage{subfig}
\usepackage{hyperref}

\usepackage[utf8]{inputenc}

\newtheorem{theorem}{Theorem}[section]
\newtheorem{lemma}[theorem]{Lemma}
\newtheorem{corollary}[theorem]{Corollary}
\newtheorem{observation}[theorem]{Observation}

\newtheorem{claim}[theorem]{Claim}
\newtheorem{definition}[theorem]{Definition}

\newcommand\cref[1]{Corollary~\ref{cor:#1}}

\begin{document}

\title{Saturation problems about forbidden $0$-$1$ submatrices}

\author{Radoslav Fulek\thanks{The University of California San Diego, La Jolla, CA, USA.}\and Bal\'azs Keszegh\thanks{Alfr\'ed R{\'e}nyi Institute of Mathematics and MTA-ELTE Lend\"ulet Combinatorial Geometry Research Group. Research supported by the Lend\"ulet program of the Hungarian Academy of Sciences (MTA), under the grant LP2017-19/2017 and by the National Research, Development and Innovation Office -- NKFIH under the grant K 132696.}}

\date{}

\maketitle
\begin{abstract}
 A $0$-$1$ matrix $M$ is \emph{saturating}  for a  $0$-$1$ matrix $P$ if $M$ does not contain a submatrix that can be turned into $P$ by changing some $1$ entries to $0$ entries, and  changing an arbitrary $0$ to $1$  in $M$ introduces such a submatrix in $M$.  In saturation problems for $0$-$1$ matrices we are interested in estimating the minimum number of $1$ entries in an $m \times n$  matrix that is saturating for $P$, in terms of $m$ and $n$.
 In other words, we wish to give good estimates for the \emph{saturation function} of $P$.
Recently, Brualdi and Cao  initiated  the study of saturation problems in the context of  $0$-$1$ matrices. 

We extend their work in several directions. We prove that every $0$-$1$ forbidden matrix has its  saturation function  either in $\Theta(1)$ or $\Theta(n)$  in the case when we restrict ourselves to square saturating matrices.  Then we give a partial answer to a question posed by  Brualdi and Cao about the saturation function of $J_k$, which is obtained from the identity matrix $I_k$ by  putting the first row after the last row.
Furthermore, we exhibit a $5\times 5$ permutation matrix with the saturation function bounded from the above by a fixed constant. We complement this result by identifying large classes of $0$-$1$ matrices with linear saturation function. Finally, we completely resolve the related semisaturation problem as far as the constant vs. linear dichotomy is concerned.
\end{abstract}

\section{Introduction}
The study of extremal questions is a very popular theme in the area of combinatorics. Broadly speaking, in an extremal problem we are usually interested in an estimation of the maximum density of combinatorial objects, such as graphs or matrices, avoiding a given substructure that we call the \emph{forbidden subgraph or submatrix}. In graphs, the density is simply the number of edges. For several extremal problems on (hyper)graphs their so-called saturation counterpart was also investigated \cite{2020saturation,dudek2013minimum,erdoshajnalmoon,ferrara2017saturation,frankl2020vc,Gerbner2013,keszegh2020induced,kaszonyituza}. In a \emph{saturation} problem, instead of the maximum density we are interested in estimating  the smallest (\emph{critical}) number of (hyper)edges 
such that adding an arbitrary (hyper)edge introduces the forbidden sub(hyper)graph. 
A variant of saturation problems is the so-called \emph{semisaturation} (sometimes called strong saturation or oversaturation) problem \cite{furedikim}, in which we are interested in the smallest (\emph{critical}) number of (hyper)edges such that adding an arbitrary (hyper)edge introduces a new copy of the forbidden sub(hyper)graph. 

Extremal problems on ordered bipartite graphs translate to extremal problems on $0$-$1$ matrices, which we study in the present note. Before giving a small history of this area, we first define the three functions that we are concerned with.

\bigskip
{\bf Extremal function.}
A $0$-$1$ matrix $M$ \emph{contains} a \emph{pattern} $P$, which is a not-all-zeros $0$-$1$ matrix, if $M$ contains a submatrix $P'$ that can be transformed into the matrix $P$ by replacing some (potentially none) $1$ entries with $0$ entries. The matrix $P'$ is an \emph{occurrence} of $P$ in $M$.  A $0$-$1$ matrix $M$ \emph{avoids} a pattern $P$ if $M$ does not contain $P$.
For a $0$-$1$ matrix $M$, the weight of $M$ is the number of $1$ entries of $M$. Let $ex(P,m,n)$ denote the maximal weight of a $0$-$1$ matrix of size $m\times n$ that avoids $P$ as a submatrix. We denote $ex(P,n)=ex(P,n,n)$.

Usually, $P$, the matrix to be avoided, is referred to as a forbidden or excluded pattern. Square-shaped patterns that have exactly one $1$ entry in every row and column are called permutation patterns. 

\bigskip
{\bf Saturation function.}
A matrix $M$ is \emph{saturating} for a pattern $P$ if it avoids $P$ as a submatrix and is maximal with this property, that is, if one changes any $0$ entry to a $1$ entry in the matrix $M$ then the resulting matrix $M'$ contains $P$. Let $sat(P,m,n)$ denote the minimal weight of a $0$-$1$ matrix $M$ of size $m\times n$ saturating for $P$. We denote $sat(P,n)=sat(P,n,n)$. The matrix $M$ is \emph{saturating} for $P$.

\bigskip
{\bf Semisaturation function.}
Let $ssat(P,m,n)$ denote the minimal weight of a $0$-$1$ matrix $M$ of size $m\times n$ such that if one changes any $0$ entry to a $1$ entry in the matrix $M$ then the resulting matrix $M'$ contains a new copy of $P$. Note that we do not require that $M$ avoids $P$ as a submatrix. We denote $ssat(P,n)=ssat(P,n,n)$.  The matrix $M$ is \emph{semisaturating} for $P$.

\begin{observation}
\label{obs:mon}
 By definition it always holds that
	$$ssat(P,m,n)\le sat(P,m,n)\le ex(P,m,n)$$ and if $sat(P,m,n)=ex(P,m,n)$ then every maximal matrix avoiding $P$ has the same weight.
\end{observation}

\bigskip
{\bf History of the extremal problem.}
The study of the asymptotic behavior of $ex(P,n)$ arose in the  running time analysis of a shortest path geometric algorithm due to Mitchell~\cite{mitchell19921}. Soon after a strong impetus for the development of this theory came from discrete geometry~\cite{furedi1990maximum,furedi1992davenport}. The main highlight of the subsequent developments is  
a proof of the famous Stanley-Wilf conjecture~\cite{marcus2004excluded} by Marcus and Tardos via Klazar's argument~\cite{klazar2000furedi} showing that the conjecture holds once we prove that $ex(P,n)=O(n)$
for all permutation matrices $P$.  Although   
the order of magnitude of $ex(P,n)$ was researched extensively for many families of patterns $P$, see e.g.,~\cite{fulek2009linear,geneson2009extremal,keszegh2009linear,marcus2004excluded,pach2006forbidden,pettie2011degrees,tardos20050},
the asymptotic behavior of $ex(P,n)$ for  natural families of patterns such as trees is still far from being well-understood.

In the study of $ex(P,n)$ two (overlapping) main tendencies can be pinpointed, first determining $ex(P,n)$ for specific matrices, e.g., in \cite{tardos20050} all patterns with weight at most $4$ and in~\cite{fulek2009linear,keszegh2009linear} further patterns of weight $5$ are considered. It turns out that the extremal function is always at least linear but can have various bigger growth as well, such as $\Theta(n\log n)$ and $\Theta(n^{3/2})$. Second, it is easy to see that the growth rate is at least linear and at most quadratic and thus it gained a considerable attention to determine the patterns with linear or quasi-linear extremal functions. In this quest the most notable result is that $ex(P,n)=\Theta(n)$ for all permutation patterns \cite{marcus2004excluded}, which was generalized to patterns that we get from permutation patterns by multiplying columns \cite{geneson2009extremal,keszegh2009linear}, where they also show that this implies that there exist infinite many containment-minimal non-linear patterns. Other useful tools are various operations which can extend or combine patterns with linear extremal functions to get new such patterns~\cite{furedi1990maximum,tardos20050,keszegh2009linear}. Recently, patterns with quasi-linear extremal functions gained attention, it is conjectured that all patterns that are incidence matrices of bipartite trees are such~\cite{pach2006forbidden,trees}.

\bigskip
{\bf History of the saturation problem.}
While extremal problems on $0$-$1$ matrices have a long history, only recently Brualdi and Cao \cite{brualdi2020pattern} initiated the study of the  saturation problems for $0$-$1$ matrices,
that is, the study of $sat(P,m,n)$ for various patterns $P$. 
In \cite{brualdi2020pattern} it is proved that all maximal matrices avoiding the identity matrix $I_k$ have the same weight:
\begin{theorem}\label{thm:ik}\cite{brualdi2020pattern}
	 $sat(I_k,m,n)=ex(I_k,m,n)=(k - 1)(m + n - (k - 1))$ (if $m, n \ge k$).
\end{theorem}
 They also give a structural description of maximal families. 
 Note that if $m<k$ or $n<k$ then no $m\times n$ matrix can contain $I_k$ and thus the only saturating matrix is the all-$1$ matrix and so $sat(I_k,m,n)=ex(I_k,m,n)=mn$ trivially.
Further, they consider the permutation pattern $J_k$ which we get from $I_k$ by putting its first row after the last row (see Figure~\ref{fig:ikjk}). (Note that we do not denote by $J_k$ an all 1 matrix which is quite common.) They show $sat(J_3,m,n)=ex(J_3,m,n)=2(m+n-2)$ and that $sat(J_k,m,n)\le(k -1)(m+
n - (k - 1))$ and conjecture that $sat(J_k,m,n)=ex(J_k,m,n)=(k -1)(m+
n - (k - 1))$ for $k>3$ as well. 

\bigskip
{\bf Our results about saturation.}
We continue this line of research. While the above specific results would perhaps give the impression that $sat(.)$ behaves always the same way or very similar to $ex(.)$, we show multiple evidence that this is not the case. The following theorem shows that compared to the growth rate of $ex(P,n)$, $sat(P,n)$ is always at most linear and exhibits a much simpler dichotomic behavior:

\begin{theorem}\label{thm:dichotomy}
	For any $k\times l$ pattern $P$, and for any fixed $n_0,m_0$, 
	\[sat(P,m,n)\le (k-1)n+(l-1)m-(k-1)(l-1),\]
	\[sat(P,n)=O(1) \text{ or } sat(P,n)=\Theta(n),\]
	\[sat(P,m_0,n)=O(1) \text{ or } sat(P,m_0,n)=\Theta(n),\]
	\[sat(P,m,n_0)=O(1) \text{ or } sat(P,m,n_0)=\Theta(m).\]
	In addition, if $sat(P,m_0,n)=O(1)$ then $sat(P,m_1,n)=O(1)$ for every $m_1>m_0$ and similarly if $sat(P,m,n_0)=O(1)$ then $sat(P,m,n_1)=O(1)$  for every $n_1>n_0$.
\end{theorem}

Note that when $n<k$ or $m<l$ then even the $m\times n$ matrix with all $1$ entries avoids $P$, thus $sat(P,m,n)=nm$, which shows that $n_0<k$ and $m_0< l$ are trivial cases of Theorem \ref{thm:dichotomy} where $sat(P,m_0,n)=\Theta(n)$ and $sat(P,m,n_0)=\Theta(m)$. Note also that the upper bound of Theorem \ref{thm:dichotomy} is attained by the identity matrix $I_k$ \cite{brualdi2020pattern}.

From now on for simplicity we concentrate only on $sat(P,n)$ and most of the times we omit variants of our results for $sat(P,m,n)$, although in many cases we could state such variants. 

Even though there are only two different classes of patterns with respect to the growth rate, it seems to be challenging to understand when does a pattern have a bounded saturation function and when a linear one. Our aim is to take the first steps in this direction.

First, it is already non-trivial to come up with a pattern that has a bounded saturation function, our first result is that there exists such a permutation pattern of weight $5$. Let $$Q=\begin{pmatrix} &\bullet&&&\\&&&&\bullet \\&&\bullet\\ \bullet\\&&&\bullet \end{pmatrix}$$ where $1$ entries are shown as dots and $0$ entries are not shown.
We know that $ex(Q,n)=\Theta (n)$. 

We show that $sat(Q,n)$ differs substantially from $ex(Q,n)$, as it is bounded.

\begin{theorem}\label{thm:5x5}
	The pattern $Q$ has $sat(Q,n)< 400$. 
\end{theorem}


In contrast, we exhibit a wide class of patterns for which the saturation function is linear:

\begin{claim}\label{claim:twoineveryrow}
	If every row (or column) of a pattern $P$ contains at least two $1$-entries then $sat(P,n)=\Theta(n)$.
\end{claim}

\begin{theorem}\label{thm:emptyfirstrow}
	If the first (or last) row (or column) of a pattern $P$ is all-$0$ then $sat(P,n)=\Theta(n)$.
\end{theorem}

By Observation~\ref{obs:mon}, both Claim~\ref{claim:twoineveryrow} and Theorem~\ref{thm:emptyfirstrow} are implied by a stronger claim, Theorem~\ref{thm:ssat}, stated below.

Note that it is an easy statement that adding empty rows and columns to $P$ does not change the order of magnitude of $ex(P,n)$, while for $sat(P,n)$ the previous theorem shows all that we know about the effect of such an operation on the saturation function. Thus, even for such a seemingly harmless operation we can see a non-trivial behavior, which is many times the case when considering saturation problems.

\begin{theorem}\label{thm:A00B}
	If $P=\begin{pmatrix}
      A  & {\bf 0} \\ {\bf 0} & B
    \end{pmatrix}$, for some $0$-$1$ submatrices\footnote{With a slight abuse of notation, throughout the paper ${\bf 0}$ denotes an all-$0$ matrix of any size.} $A,B\not={\bf 0}$, then $$sat(P,n)=\Theta(n).$$
\end{theorem}

It is an easy observation that the function $ex(.)$ is monotone for pattern containment, that is, if from a pattern $P$ we get pattern $P'$ by replacing some $0$ entries with $1$ entries then $ex(P,n)\le ex(P',n)$. However, the function $sat(.)$ shows again a different behavior, as it does not have this natural property. Indeed, $sat(Q,n)=O(1)$, while if we delete, e.g., the leftmost $1$ entry from $Q$, we get a pattern $Q'$ which has $sat(Q',n)=\Theta(n)$ by Theorem \ref{thm:emptyfirstrow}. Also, if we replace all $0$-entries in $Q$ by a $1$ entry then we get the full-$1$ $5\times 5$ pattern $Q''$ which is easy to see that has $sat(Q'',n)=\Theta(n)$. Together, these imply that $sat(.)$ cannot be monotone increasing nor decreasing for pattern containment.

We also make progress about the saturation function of $J_k$, which was asked in \cite{brualdi2020pattern}:

\begin{theorem}\label{thm:jk}
	$sat(J_k,m,n)\ge (k-2)\max(m,n)+m+n-1-\frac{(k-2)(k-1)}{2}$ (if $m,n\ge k$).
\end{theorem}

Notice that similar to the identity matrix, if $m<k$ or $n<k$ then $sat(J_k,m,n)=mn$. Also, $sat(J_k,m,n)\le (k -1)m+(k -1)n - (k - 1)^2$ by the trivial construction of putting $1$ entries in the first $k-1$ rows and last $k-1$ rows \cite{brualdi2020pattern}, which for the case $n=m$ is about twice as much as our lower bound \footnote{More precisely, if we substract $m+n-1$ from both the lower bound and the upper bound, which is the size of the maximal staircase present in every matrix saturating for $J_k$ (see the proof of Theorem \ref{thm:jk}), then the lower bound becomes exactly half of the upper bound.}. Due to the same construction $ex(J_k,m,n)\ge (k -1)m+(k -1)n - (k - 1)^2$ while $ex(J_k,m,n)=O(m+n)$ by \cite{marcus2004excluded}. We did not try to get an explicit upper bound from their proof as anyway it is conjectured that the lower bound is the right answer \cite{brualdi2020pattern}.

\begin{theorem}\label{thm:extendcorner}
		Let $P'=\begin{pmatrix}
			A   & {\bf 0}\\ {\bf 0} & 1	\end{pmatrix}$ and $P=\begin{pmatrix}
		A  & {\bf 0} & {\bf 0} \\ {\bf 0} & 1 &  0\\ {\bf 0}& 0 & 1	\end{pmatrix}$
	, for some $0$-$1$ submatrix $A$ (that is, we get $P'$ from $A$ by adding a new column and row at the end with a single $1$ entry in their intersection and we get $P$ from $P'$ the same way), then $$ex(P,m,n)=ex(P',m-1,n-1)+m+n-1,$$ $$sat(P,m,n)=sat(P',m-1,n-1)+m+n-1.$$
\end{theorem}

First, applying this theorem repeatedly to $P'=I_1$ (note that $sat(I_1,m,n)=ex(I_1,m,n)=0$ trivially) we get that $sat(I_k,m,n)=ex(I_k,m,n)=(k - 1)(m + n - (k - 1))$ (if $m,n\ge k$), which gives another proof of Theorem \ref{thm:ik}.

By repeated application of Theorem \ref{thm:extendcorner} we can also generalize Theorem \ref{thm:jk}. Let $I'_k$  be the matrix we get from $I_k$ by reversing the order of its rows, that is, $I'_k$ has $1$ entries in its main skew diagonal and nowhere else (see Figure \ref{fig:ikjk}).

\begin{corollary}\label{cor:ikik}
	If $P$ has form $\begin{pmatrix}
		{\bf 0} & I_{k-1}  \\ I'_{l+1} & {\bf 0}
	\end{pmatrix}$, then
		$$sat(P,m,n)\ge (k-2)\cdot\max(m,n)+(l+1)(m+n)-O(1).$$
\end{corollary}

Notice that in Corollary \ref{cor:ikik} by setting $l=0$ we get $P=J_k$.

\bigskip
{\bf Semisaturation result.} The semisaturation problem turns out to be considerably simpler than the saturation problem. A dichotomy holds also in this case, but here we were able to characterize the patterns that have bounded semisaturation function.

\begin{theorem}\label{thm:ssat}
	Given a pattern $P$, $ssat(P,n)=O(1)$ if and only if all the following properties hold for $P$: 
	\begin{enumerate}
		\item The first and last row of $P$ both contain a $1$ entry that is the only $1$ entry in its column,
		\item The first and last column of $P$ both contain a $1$ entry that  is the only $1$ entry in its row,
		\item $P$ contains a $1$ entry that is the only $1$ entry in its row and column,		
	\end{enumerate}
 and $ssat(P,n)=\Theta(n)$ otherwise.
\end{theorem}

Observe that due to Observation~\ref{obs:mon}, Theorem \ref{thm:ssat} implies Claim~\ref{claim:twoineveryrow} and Theorem~\ref{thm:emptyfirstrow}, and  that every permutation pattern (including $I_k$ and $J_k$) has bounded semisaturation function.

\section{Saturation}

\begin{figure}[h]
\centering
\subfloat[The construction of a saturating matrix for $Q$. The white corner areas contain at least the $1$ entries at the locations of the dots.]{
\includegraphics[scale=0.7]{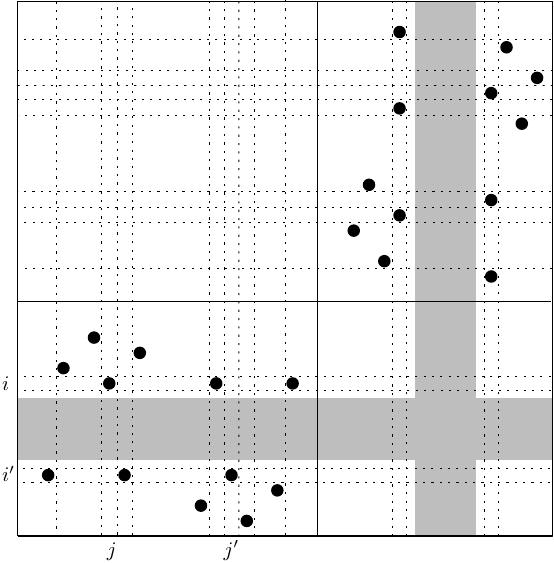}\label{fig:left}}
\hspace{2em}
\subfloat[Different occurrences of $Q$ after we add a $1$ entry into the gray horizontal strip. We distinguish the occurrences by color.]{
\includegraphics[scale=0.7]{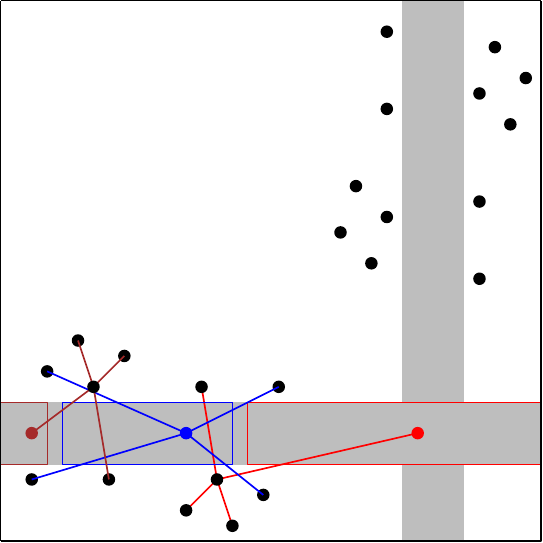}\label{fig:right}}
\caption{Saturating matrix $M$ in the proof of Theorem~\ref{thm:5x5}.}
\end{figure}

\begin{proof}[Proof of Theorem \ref{thm:5x5}]
Recall that 
 $Q=\begin{pmatrix} &\bullet&&&\\&&&&\bullet \\&&\bullet\\ \bullet\\&&&\bullet \end{pmatrix}$.
It is enough to give a construction of a saturating matrix $M$ which we give in Figure~\ref{fig:left}. The gray strips in the figure are arbitrarily wide and the 4 white corner areas contain at most a constant number of $1$ entries. We obtain the actual construction by replacing some of the $0$ entries in the corner white areas by $1$ entries so that the resulting matrix does not contain an occurrence of $Q$ and turning any additional $0$ entry in the white areas into a $1$ entry would introduce an occurrence of $Q$. 

It remains to check that the construction of the saturating matrix is correct.   Figure~\ref{fig:right} shows that the introduction of a $1$ entry in the horizontal gray strip introduces an occurrence of $Q$. This takes care also of the vertical strip by symmetry.
It remains to check that the matrix in Figure~\ref{fig:left} does not contain an occurrence of $Q$ which can be carried out by the following analysis.

Refer to Figure~\ref{fig:left}.
First, we observe that an  occurrence $Q'$ of $Q$ in $M$ must be contained in one of the two opposite ``quadrants'' (containing all the $1$ entries) determined by the pair of thin full lines crossing inside of $M$ in the figure.
Due to symmetry, it is enough to consider the case when $Q'$ is contained in the  bottom-left quadrant of $M$ denoted by $M'$. We claim that there exists two $1$ entries in the bottom-left quadrant that could possibly be the $1$ entry in the 3rd row and the 3rd column of $Q'$. Namely, $m_{ij}$ and $m_{i'j'}$ as depicted in the figure. Indeed, for every  such $1$ entry $m_{lk}$, there must exists in $M'$ four $1$ entries, one in every quadrant centered at $m_{lk}$, that is, four $1$ entries $m_{l_0k_0},m_{l_1k_1},m_{l_2k_2}$, and $m_{l_3k_3},$  such that $l>l_0$ and $k>k_0$, $l<l_1$ and $k>k_1$, $l>l_2$ and $k<k_2$, and $l<l_3$ and $k<k_3$, respectively. However, neither of these two $1$ entries can be contained in an occurrence of $Q$ by a simple case analysis.

We refrain from optimizing the exact value of  $sat(Q,n)$. The upper bound of $400$ in the statement of the theorem is obtained by observing that the white areas form a $20\times 20$ submatrix and in the construction $1$ entries are only in the white areas.
\end{proof}

%

\begin{proof}[Proof of Theorem \ref{thm:A00B}]
Let $M$ denote a $0$-$1$ matrix of size $n\times n$ saturating for $P$, that is it avoids $P$ but turning any $0$ entry of $M$ into $1$ results in a matrix $M'$ that contains $P$. We prove that $M$ cannot have an all-zero row or column which clearly proves the theorem. Recall that  $P=\begin{pmatrix}
      A  & 0 \\ 0 & B
    \end{pmatrix}$.
    
For the sake of contradiction suppose that $M=(m_{ij})$ contains  
an all-zero row (the proof of the non-existence of an all-zero column is analogous.)
Let the $l$-th row of $M$ be all-zero.
By turning $m_{lk}$  to 1, for an arbitrary $k, \ 1\le k\le n$, we obtain the matrix $M_k$ containing $P$. Note that $m_{l1}$ is contained in an occurrence of $A$ in $M_1$ that is the submatrix of an occurrence of $P$ while this is not true for $m_{ln}$.

It follows that there exists a maximum index $1\le k'<n$ such that 
 $m_{lk'}$ is contained in an occurrence of $A$
in $M_{k'}$ that is a submatrix of an occurrence $P'$  of $P$ in $M_{k'}$. Let $B'$ denote the occurrence of $B$ contained in $P'$.
Then $m_{l(k'+1)}$ is contained in an occurrence of $B$ in $M_{k'+1}$ that is a submatrix of an occurrence $P'' $  of $P$ in $M_{k'+1}$.  Let $A''$ denote the occurrence of $A$ contained in $P''$.
 
Note that the occurrence $A''$ of $A$ and $B'$ and $B$ are submatrices of an occurrence of $P$ in $M$, a contradiction. 
 \end{proof}

\begin{proof}[Proof of Theorem \ref{thm:dichotomy}]
For a pattern $P$ with all $0$'s, trivially $sat(P,m,n)=0$. Now let $P$ be a $k\times l$ (i.e., has $k$ rows and $l$ columns) not all-$0$ pattern, having a $1$ entry in the intersection of the $k'$th row and $l'$th column. 

Now let $M$ be the $m\times n$ matrix which has all $1$ entries in its first $k'-1$ and last $k-k'$ rows, first $l'-1$ and last $l-l'$ columns. It is easy to see that $M$ is saturating for $P$. Counting the $1$ entries in $M$ we get that $sat(P,m,n)\le (k-1)n+(l-1)m-(k-1)(l-1)$, showing that $sat(P,n,n), sat(P,m,n_0), sat (P,m_0,n)$ can grow at most linearly.

Next we prove that these functions are either bounded or have linear growth.

Let $k'=\max(k,l)$. Assume that $sat(P,n_0,n_0)< \frac{n_0}{k'-1}$ for some $n_0\ge k'-1$, we want to show that in this case $sat(P,n,n)\le sat(P,n_0,n_0)$ for $n\ge n_0$. Indeed, in this case there exists an $n_0\times n_0$ size saturating matrix $M_0$ ($n_0\ge k'$) with weight less than $\frac{n_0}{k'-1}$. Then there must be $k-1$ consecutive empty rows and $l-1$ consecutive empty columns in $M$. For any $n\ge n_0$ let $M$ be the $n\times n$ size matrix which we get from $M_0$ by replacing these empty rows and columns with as many new empty rows and columns as necessary.

We claim that $M$ is a saturating matrix as well. First, it cannot contain a copy of $P$ as (using that $P$ is non-empty) this copy uses at most $k-1$ new empty rows and $l-1$ empty columns. This implies that $P$ can already be found in $M_0$, a contradiction. On the other hand we claim that $M$ is a maximal matrix avoiding $P$. Indeed, if we could replace a $0$ entry with a $1$ in $M$ then we could find a corresponding $0$ entry in $M_0$ with the same property, contradicting that $M_0$ is a saturating matrix\footnote{We note that for maximality we do not need multiple consecutive empty rows and columns, one of each would be enough. Multiple empty rows and columns are only needed to show that $M$ avoids $P$.}. 

Similarly (by multiplying either only the empty rows or the empty columns) we can see that if $sat(P,m_0,n_0)< \frac{n_0}{k-1}$ for some $n_0\ge k-1$ and $m_0\ge l-1$, then $sat(P,m_0,n)\le sat(P,m_0,n_0)$ (for $n\ge n_0$) and if $sat(P,m_0,n_0)< \frac{m_0}{l-1}$ for some $n_0\ge k$ and $m_0\ge l$, then $sat(P,m,n_0)\le sat(P,m_0,n_0)$ (for $m\ge m_0$).

These together imply the theorem.
\end{proof}

\section{Saturation of $J_k$}
\begin{figure}
$$I_4=\begin{pmatrix}\bullet&&&\\&\bullet&&&\\&&\bullet \\&&&\bullet\end{pmatrix}~~I'_4=\begin{pmatrix} &&&\bullet\\&&\bullet\\&\bullet \\\bullet \end{pmatrix}~~J_5=\begin{pmatrix} &\bullet&&&\\&&\bullet \\&&&\bullet\\ &&&&\bullet\\\bullet \end{pmatrix} ~~ J'_5=\begin{pmatrix} &&&\bullet\\&&\bullet \\&\bullet\\ \bullet\\&&&&\bullet \end{pmatrix}$$
\caption{Examples for $I_k,I'_k,J_k$ and $J'_k$.}
\label{fig:ikjk}
\end{figure}

See the examples on Figure \ref{fig:ikjk} for the following definition.

\begin{definition}
	$I_k$ denotes the identity matrix of size $k\times k$. $I'_k$ denotes the matrix we get from $I_k$ by reflecting it vertically.
	
	$J_k$ is the matrix we get from $I_k$ by cyclically shifting its rows so that its first row follows the last row. $J_k'$ is the matrix we get from $J_k$ by reflecting it vertically. 
\end{definition}

\begin{definition}
	A \emph{staircase} in a matrix $M$ is a set $S$ of positions of $M$ such that the top-right and bottom-left positions of $M$ are in $S$ and for every position $p$ in $S$ the following two properties hold:
	\begin{itemize}
		\item exactly one of the neighboring positions above $p$ and right to $p$ is in $S$ except when $p$ is the top-right position,
		\item exactly one of the neighboring positions below $p$ and left to $p$ is in $S$ except when $p$ is the bottom-left position.
	\end{itemize}
\end{definition}

Note that a staircase in an $m\times n$ matrix $M$ has exactly $m+n-1$ positions.

\begin{lemma}\label{lem:staircase}
	Given a non-$(1\times 1)$ pattern $P$ in which the last row and column both contain exactly one $1$ entry, which is in their intersection. Then in any matrix $M$ saturating for $P$ there is a staircase $S$ in $M$ such that all positions in $S$ contain a $1$ entry and all positions that are below $S$ contain only $0$ entries.
\end{lemma}

\begin{proof}
	Let $m,n\ge k$ and the $m\times n$ matrix $M$ be saturating for $P$.
	Let $S$ be the set of positions of $1$ entries of $M$ for which there is no $1$ entry in $M$ that  has a larger row index and bigger column index. We claim that $S$, the set of these extremal $1$ entries, is a staircase.
	
	First observe that  positions $(m,1)$ and $(1,n)$ in $M$ must be $1$ entries and so these are in $S$.
	
	Let $p$ be a $1$ entry of $S$ different from $(m,1)$.  We need to show that exactly one of the neighboring positions $p'$ left to $p$ and $p''$ below $p$ is in $S$. Both of them cannot be in $S$ by definition of $S$. For the sake of contradiction, we assume that none of them is in $S$.
	
	First, we consider the case when $p$ is the lowest $1$ entry in its column, in which case $p$ is not in the first column. Then we use the saturation property for $p'$ to get an occurrence of $P$, in which $p'$ must be the bottom-right $1$ entry (by the definition of $S$) and then replacing $p'$ with $p$ we get an occurrence of $P$ in $M$ (contradiction).
	
	Second, we consider the case when $p$ is not the lowest $1$ entry in its column. Then there is a $1$ entry $q$ below $p''$. We use the saturation property for $p''$ to get an occurrence of $P$ in which $p''$ must be the bottom-right $1$ entry (by the definition of $S$) and then replacing $p''$ with $q$ we get an occurrence of $P$ in $M$ (contradiction).
	
	Symmetrically it follows that if $p$ be a $1$ entry of $S$ different from $(1,n)$ then exactly one of the neighboring positions $p'$ right to $p$ and $p''$ above $p$ is in $S$.
\end{proof}

By symmetry, whatever bounds we show about saturating $J_k'$, the same holds for saturating $J_k$. Let $M$ be an $m\times n$ matrix saturating for  $J_k'$. We need to show that for $w(M)$, the weight of $M$, we have $w(M)\ge  n(k-1)+m-1-\frac{(k-2)(k-1)}{2}$ as then  $w(M)\ge m(k-1)+n-1-\frac{(k-2)(k-1)}{2}$ will follow by symmetry, together implying the required lower bound.

Let $S$ be the staircase in $M$, whose existence is guaranteed by applying Lemma \ref{lem:staircase} with $P=J'_k$. There are exactly $m+n-1$ positions in $S$ all of which are $1$ entries, and the rest of the $1$ entries of $M$ are above $S$.

\begin{definition}
\label{def:level}
Let $p$ be a position in $M$ above $S$.
Let the matrix $M_p=M$ if $p$ is a $1$ entry, otherwise let $M_p$ be a matrix  that  we get from $M$ by turning $p$ into a  $1$ entry.
The position $p$ in $M$ has \emph{level} $l$,  $l(p)=l$ for short, if the following holds. The value of $l$ is the largest natural number such that there exists an occurrence of $I_l'$ in $M_p$ which lies completely above $S$ (that is, every $0$ and $1$ entry of the copy of $I_l'$ is above $S$) and $p$ is its bottom-left $1$ entry. A $1$ entry of $M$ has/is on level $l$ if its position has level $l$. Positions on $S$ and below $S$ have no level.
\end{definition}

\begin{lemma}\label{obs:monotone}
For the  matrix $M$ which is saturating for $J_k'$,	the following hold:
	\begin{itemize}
		\item[(i)] The level of every $1$ entry is at most $k-2$.
		\item[(ii)] The level of every $0$ entry is at most $k-1$.
		\item[(iii)] For every row $i\le m-1$, the first position in this row is on level $i$ if $i\le k-2$, and  otherwise on level $k-2$  or $k-1$.
		\item[(iv)] For every row $i\le m-1$, the last position in this row which is above $S$ is on level $1$.
		\item[(v)] If two positions above $S$ are next to each other in a row, $p'$ left to $p$, then $l(p)\le l(p')\le l(p)+1$.		\end{itemize}
\end{lemma}

\begin{proof}
	Notice that every position $(m,j)$ and $(j,n)$, such that $1\le j\le k-1$, is a $1$ entry, which then must be in $S$ and also every position $(i,j)$, such that $i+j\le k-1$, is a $1$ entry. Indeed, otherwise exchanging a $0$ entry in one of these positions to a $1$ entry should introduce a copy of $J_k'$, which is impossible. These imply $(iii)$.
	
	A $0$ entry with level at least $k$ or a $1$ entry with level at least $k-1$  would mean the existence of a copy of $I'_{k-1}$ in $M$ completely above $S$ which together with a $1$ entry on $S$ would be a copy of $J'_k$, a contradiction. These imply $(i)$ and $(ii)$. 
	
	For any position $p$ which is last in its row above $S$ it is true that the position to the right of $p$ is on $S$. Thus, a maximum size square matrix whose bottom-left position is $p$ and is completely above $S$ has size $1\times 1$. This implies $(iv)$.
	
	Finally, if a position $p$ is on level $l(p)$ then there is a copy of $I_{l(p)}$ in $M_p$ above $S$. This copy is in the intersection of the set or rows $R$ and set of columns $C$. Now if in $C$ we exchange the first column (note that  $p$ is in this column) with the column containing $p'$ (the position left to $p$) to get $C_1$, then in the intersection of $C_1$ and $R$ there is a copy of $I_{l(p)}$ in $M_{p'}$ above $S$, showing that $l(p')\ge l(p)$.
	
	Conversely, if $p'$ is on level $l(p')$ then there is a copy of $I_{l(p')}$ in $M_{p'}$ above $S$. This copy is in the intersection of the set or rows $R'$ and set of columns $C'$. Now if in $C'$ we exchange the first column (note that  $p'$ is in this column) with the column containing $p$ to get $C_2$ (if the column of $p$ was already in $C$ then we simply delete the column of $p'$ to get $C_2$), then in the intersection of $C_2$ and an appropriate subset of $R'$ of size $l(p')-1$ there is a copy of $I_{l(p')-1}$ in $M_{p}$ above $S$, showing that $l(p)\ge l(p')-1$. This implies $(v)$.
\end{proof}

\begin{figure}[ht]
	\centering	
		\includegraphics[scale=0.8]{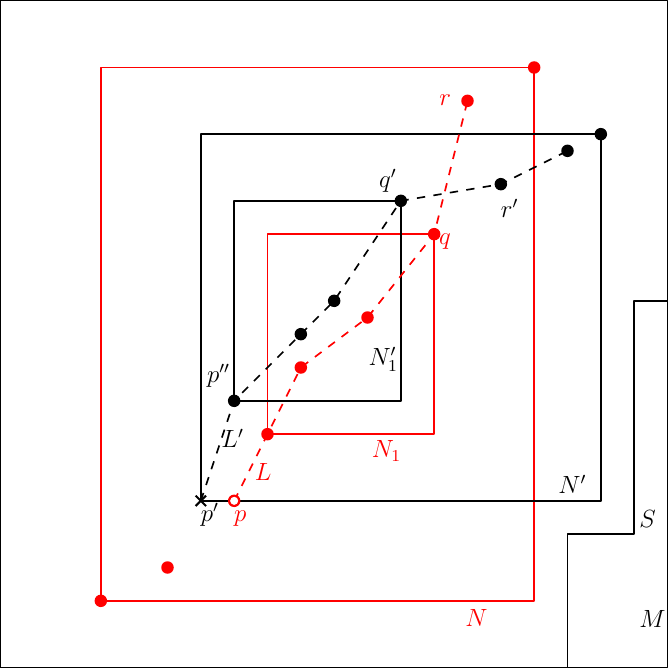}
	\caption{Intersecting $L$ and $L'$ from the proof of Lemma \ref{lem:level}.}	
			\label{fig:jk}
\end{figure}

\begin{figure}[ht]
	\centering	
	\includegraphics[scale=0.6]{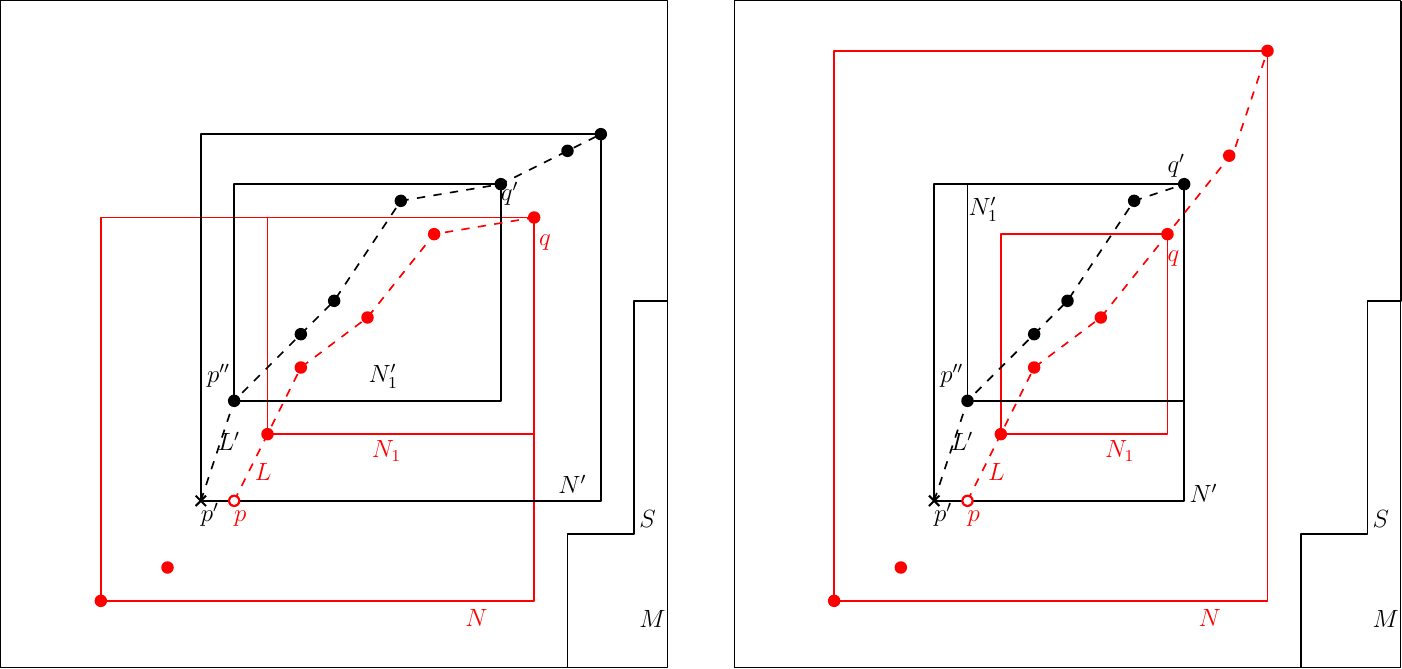}
	\caption{Proof of Lemma \ref{lem:level}.}	
	\label{fig:jk_nonint}
\end{figure}

It was noticed in 2023 by Shen-Fu Tsai \cite{TsaiPC} that the proof of the following lemma has gaps in the published version \cite{01sat}. Therefore, here we present a corrected proof following his idea.

\begin{lemma}\label{lem:level}
	For every $i\in \mathbb{N}$ and $l\in \mathbb{N}$ such that $1\le l\le k-2$ and $l\le i\le m-1$, the leftmost position on level $l$ in the $i$th row of $M$ is a $1$ entry.
\end{lemma}

\begin{proof}		
	Fix $i,l$ and let $p$ be the leftmost position on level $l\le k-2$ in row $i$. See Figure \ref{fig:jk} and Figure \ref{fig:jk_nonint} for the rest of the proof.
	We need to show that $p$ is a $1$ entry.
	
	Assume first that $p$ is the first position in row $i$. To reach a contradiction, assume $p$ is a $0$ entry. Then as $M$ is a saturating matrix, changing $p$ to a $1$ entry must introduce a copy of $J'_k$. As $p$ is in the first column, the only possibility for that is that the $1$ entry in $p$ is the leftmost $1$ entry in this copy of $J'_k$, and without loss of generality, the bottom-left $1$ entry in this copy of $J'_k$ is on $S$. This implies that $p$ is at least on the $(k-1)$th level, which contradicts that $p$ is on level $l$ with $l\le k-2$. Thus, $p$ must be a $1$ entry, as claimed.
	
	From now on we can assume that $p$ is not the first position in its row. To reach a contradiction, assume again that $p$ is a $0$ entry. Let the position immediately left to $p$ be $p'$. By the choice of $p$ we have $l(p')=l+1$ while $l(p)=l$. This implies that there must be an occurrence $N''$ of $I'_{l}$ in $M$ above $S$, whose leftmost $1$ entry is in a position $p''$, that is in the same column as $p$, and is above $p$ (not necessarily in the previous row). In addition, $N''$ extended with $p'$ forms an $N'$ that lies completely above $S$ and becomes a copy of $I'_{l+1}$ if we put a $1$ entry in $p'$. Furthermore, as $M$ is a saturating matrix, there must be an occurrence $N$ of $I'_{k-2}$ in $M$ above $S$ which could be extended to a copy of $I'_{k-1}$ by putting a $1$ entry in $p$. Notice that $p$ cannot be the position of the leftmost $1$ entry in $N$ as then the level of $p$ would be $k-1$, a contradiction.
	
	Connecting the $1$ entries one-by-one from left to right in $N$ starting from $p$ we get a polygonal path $L$. Connecting the $1$ entries one-by-one from left to right in $N'$ starting from $p'$ we get a polygonal path $L'$. Note that the first segment of $L'$, $p'p''$, is disjoint from $L$. 

 In the following paragraph we define $1$ entries $q\in L$ and $q'\in L'$ depending on the relative positions of the rightmost entries of $L$ and $L'$.
 
	Assume first that these two polygonal paths intersect. Then let $qr\in L$ and $q'r'\in L'$ be two segments that intersect such that $q$ is closer than $r$ to $p$ on $L$ and $q'$ is closer than $r'$ to $p'$ on $L'$, see Figure~\ref{fig:jk}.	Assume next that the two polygonal paths do not intersect. First, if the top-right vertex of $N'$ is to the right of the top-right vertex of $N$, then let $q$ be the top-right vertex of $N$ and let $q'$ be the rightmost vertex of $N'$, that is not to the right to $q$, see the left side of Figure~\ref{fig:jk_nonint}. Otherwise, if the top-right vertex of $N'$ is not to the right of the top-right vertex of $N$, then let $q'$ be the top-right vertex of $N'$ and let $q$ be the topmost vertex of $N$, that is not above $q'$, see the right side of Figure~\ref{fig:jk_nonint}.
	
	Having defined $q$ and $q'$ in all cases, let $N_1$ be the $1$ entries of $N$ in the columns of $N$ to the right of the column of $p$ and to the left (including) the column of $q$. Similarly, let $N'_1$ be the $1$ entries of  $N'$ in the columns of $N'$ to the right of the column of $p'$ and to the left (including) the column of $q'$.
		
    Both $N_1$ and $N_1'$ form occurrences of a vertically reflected identity matrix, say $N_1$ is an occurrence of $I'_{s}$, and $N'_1$ is an occurrence of $I'_{s'}$ for some $s,s'$.
	
	First, assume $s'> s$, then by replacing in $N$, the sub-matrix $N_1$ by $N'_1$ we get an occurrence of $I'_{k-1}$ above $S$, a contradiction (as then $M$ contains $J'_k$). Indeed, note that there is no 1 entry of $L$ in the same column as $q'$ that is above $q'$. Here, we consider carefully all the cases in the correct order, in particular this is not possible in the case depicted by the right side of Figure~\ref{fig:jk_nonint}, since in that case $L$ and $L'$ would intersect, and hence, we would end up in the case of Figure~\ref{fig:jk}.
	
	Second, if $s'\le s$, then by replacing in $N'$, the sub-matrix $N_1'$ by $N_1$ we get an occurrence of $I'_l$ with the following property. If we introduce a $1$ entry in position $p$, then this $1$ entry is extendable to an occurrence of $I'_{l+1}$ above $S$ having its bottom-left $1$ entry in position $p$, implying that the level of $p$ is at least $l+1$, a contradiction.
\end{proof}

\begin{proof}[Proof of Theorem \ref{thm:jk}]
There are $n+m-1$ many $1$ entries in $S$. Additionally, by Lemma \ref{lem:level} for every $1\le l\le k-2$ there are $n-l$ rows with a $1$ entry on level $l$. Altogether this is at least $$n+m-1+\sum_{l=1}^{ k-2}(n-l)=n+m-1+n(k-2)-\sum_{l=1}^{k-2}l=n(k-1)+m-1-\frac{(k-2)(k-1)}{2}$$ $1$ entries in $M$, as required.
\end{proof}

Let $L_l$ denote the set of leftmost entries at level $l$ in (the rows of) $M$. By changing ``bottom-left'' to ``top-right'' in Definition~\ref{def:level}, we analogously define $L_l'$ to be the set of top-most entries at level $l$ in (the columns) of $M$.
It holds that $L_1$ and $L_{k-2}'$ (and by symmetry also $L_{k-2}$ and $L_1'$)  are disjoint (expect for the top-left entry in the case $k=3$) as otherwise for a $1$ entry $p$ in $L_1\cap L_{k-2}$ there must be a $1$ entry in $M$ above $p$ and a copy of $I'_{k-2}$ whose top-left $1$ entry is left of $p$ such that these $(k-1)$ $1$ entries together with some $1$ entry in $S$ would form an occurrence of $J_k'$, a contradiction. 

In particular, for $k=3$ in an $m\times n$ matrix $M$ avoiding $J_3'$, by counting $1$ entries using Lemma \ref{lem:level} in both $L_1$ and $L_1'$ (their only common $1$ entry being the top-left $1$ entry of $M$), similarly as in the proof of Theorem~\ref{thm:jk} and adding the $1$ entries in $S$ we get the exact bound $2(m+n-2)$ for $J_3'$ and in turn also for $J_3$, obtained also by Brualdi and Cao~\cite{brualdi2020pattern}.

\begin{proof}[Proof of Theorem \ref{thm:extendcorner}]
	
	Take any $m-1\times n-1$ matrix $M'$ saturating for $P'$ and add a new column and row to it with all $1$ entries to get the matrix $M$ with $w(M)=sat(P',m-1,n-1)+m+n-1$. It is easy to see that $M$ is saturating for $P$ (notice that so far we did not need that $P'$ is an extension of $A$, we just needed that $P$ is the extension of $P'$). First, if $M'$ is chosen to have minimal weight then we get that $sat(P,m,n)\le w(M')+m+n-1= sat(P',m-1,n-1)+m+n-1$. Second, if $M'$ is chosen to have maximal weight then we get that $ex(P,m,n)\ge w(M')+m+n-1= ex(P',m-1,n-1)+m+n-1$.
	
	We are left to prove that $ex(P,m,n)\le ex(P',m-1,n-1)+m+n-1$ and $sat(P,m,n)\ge sat(P',m-1,n-1)+m+n-1$. Let $A$, $P$ and $P'$ be matrices as defined in the theorem and $M$ be a matrix saturating for $P$. Applying Lemma \ref{lem:staircase} to $P$ we get a staircase $S$ such that it has $(m+n+1)$ $1$ entries and there exists no $1$ entry that lies below and to the right from a $1$ entry of $S$. 
	
	Now let the set of positions of $M$ above $S$ be $M_1$ and the set of positions below $S$ be $M_2$ (notice that $M_2$ has only $0$ entries). By deleting $S$ and shifting every position (and the $0$ entries in it) of $M_2$ by one to the right and by one upwards we get the set of positions $M_2'$. Then $M_1$ together with $M_2'$ forms an $m-1\times n-1$ matrix $M'$ with $w(M')=w(M)-(m+n-1)$. We claim that $M'$ is saturating for $P'$. Using this we can finish the proof. First, if $M$ is chosen to have minimal weight then we get that $sat(P,m,n)=w(M)= w(M')+(m+n-1)\ge sat(P',m-1,n-1)+m+n-1$. Second, if $M$ is chosen to have maximal weight then we get that $ex(P,m,n)=w(M)= w(M')+(m+n-1)\le ex(P',m-1,n-1)+m+n-1$.
	
	So we are left to prove that $M'$ is saturating for $P'$. First we claim that $M'$ avoids $P'$. Assume on the contrary that $M'$ contains $P'$. As $M_2'$ is empty of $1$ entries, the bottom-right $1$ entry in this copy of $P'$ must be in $M_1$ in some position $p$. Then by the definition of $S$ there is a $1$ entry in $S$ in $M$ which is to the right and below $p$. This together with the copy of $P'$ forms a copy of $P$, contradicting that $M$ was $P$-free.
	
	Second we claim that turning any $0$ entry of $M'$ into a $1$ entry introduces a copy of $P'$ in $M'$. To see this take an arbitrary $0$ entry in $M'$. First assume that it is in $M_1$. Turning this into an $1$ entry in $M$ introduces a copy of $P$ in $M$ as $M$ was saturating for $P$. Necessarily the bottom-right $1$ entry of this copy is either on $S$ or above $S$. In both cases after deleting $S$ and shifting $M_2$ we still have a copy of $P'$ in $M'$, as required. Finally, if this $0$ entry is in $M_2'$ then again turning the respective $0$ entry into a $1$ entry in $M_2$ we get a copy of $P$ in $M$. In this copy the $1$ entry of $P$ which is in its penultimate row and column must be either on $S$ or above $S$. In both cases after deleting $S$ and shifting $M_2$ we still have a copy of $P'$ in $M'$, as required.
\end{proof}

\section{Semisaturation}

\begin{proof}[Proof of Theorem \ref{thm:ssat}]
	We call the $1$ entries required by the properties \emph{special}.
	First we show that if all the required properties hold for a $k\times l$ pattern $P$ then indeed $ssat(P,n)=O(1)$. If $n\le \max(2k-2,2l-2)$ then $M$ can simply be the all-$1$ $n\times n$ matrix. Otherwise, let $M$ be the matrix that has a $(k-1)\times (l-1)$ matrix with all $1$ entries in all four of its corners and has $0$ entries everywhere else. We claim that $M$ semisaturates $P$, that is, exchanging any $0$ entry in $M$ to a $1$ entry introduces a new copy of $P$. This shows that if all properties hold for $P$ then $ssat(P,n)\le (2k-2)(2l-2)=O(1)$. 
	
	Let us first assume that the new $1$ entry is in the first $k-1$ rows, that is, the coordinates of it are $(x,y)$ where $x\le k-1$ (and thus $l-1<y<m-l+1$). In this case it is easy to see that there is a new copy of $P$ in $M$ in which this new $1$ entry plays the role of the special $1$ entry required by the properties in the first row of $P$. Similarly, if the new $1$ entry is in the last $k-1$ rows or first $l-1$ columns or last $l-1$ columns then a symmetric argument works (note that the properties are also symmetric).
	Second, assume that the new $1$ entry has coordinates $(x,y)$ where $k-1<x<n- k+1$ and $l-1<y<m-l+1$. Then again there is a new copy of $P$, in which the new $1$ entry plays the role of the $1$ entry required by the last property.
	
	Now we show that if $P$ does not have at least one of the properties required, then $P=\Theta (n)$. First, by Theorem \ref{thm:dichotomy} we know that $ssat(P,n)\le sat(P,n)=O(n)$.\footnote{Note that the semisaturation function of any $k\times l$ pattern $P$ with a single $1$ entry attains the upper bound of Theorem \ref{thm:dichotomy}, that is, $ssat(P,m,n)=sat(P,m,n)=(k-1)n+(l-1)m-(k-1)(l-1)$ for such a pattern.} Let $M$ be a semisaturating matrix. First, assume that there is no special $1$ entry in the first row of $P$. Then we claim that every column of $M$ contains a $1$ entry. Indeed, if there would be a column with no $1$ entries then exchanging its first (top) $0$ entry with a $1$ entry could not introduce a new copy of $P$ as in such a copy this new $1$ entry would have to be special of this type. Thus in this case the weight of $M$ is at least $n$. If there is no special $1$ entry in the last row or first column or last column of $P$ then a symmetric argument shows that the weight of $M$ is at least $n$. Finally, if there is no special $1$ entry which is the single $1$ entry in its row and column then $M$ cannot contain both an empty row and an empty column. Indeed, in this case adding a new $1$ entry in the intersection of the empty row and empty column could not introduce a new copy of $P$ as in such a copy this new $1$ entry would have to be special of this type. Thus, the  weight of $M$ is at least $n$ in this case as well. We conclude that if any of the properties fail then $ssat(P,n)\ge n$ and thus $ssat(P,n)=\Theta(n)$.
\end{proof}

\section{Discussion}

We proved  constant vs. linear dichotomy for the saturation function in the context of $0$-$1$ matrices.
On the one hand, we identified large classes of forbidden matrices with linear saturation functions. 
On the other hand, up to left-to-right and bottom-to-top symmetry we discovered only one  
matrix, the permutation matrix $Q$, whose saturation function is in $\Theta(1)$. The only other matrix with the constant saturation function that we are aware of is the $1\times 1$ identity matrix $I_1$.

We suspect that there exist many more forbidden matrices with a constant saturation function, but we
were unable to find any such matrix besides $Q$ and $I_1$. Here, a natural problem is to characterize
forbidden permutation matrices with a linear/constant saturating function. To the best of our knowledge, this is open already if we restrict ourselves to $4\times 4$ permutation matrices.
Here, an interesting question is to decide whether  $Q$ becomes linear after the removal of the 3rd column and 3rd row from it.

Finally, it is not known if there exists an algorithm
that always stops that decides for a given $0$-$1$ matrix $P$ whether $sat(P,n)=O(1)$. 
To this end, notice that similarly as in the case of the pattern $Q$, if a pattern $P$ contains at most $s$ empty rows and empty columns (e.g., in case of $Q$ we have $s=0$), then if a matrix $M$ is saturating for $P$ and contains $s+1$ consecutive empty rows and columns, then any matrix $M'$ that we get from $M$ by further multiplying these empty rows and columns must also be saturating for $P$ and thus $P$ has a constant saturating function. Let us call such a matrix $M$ a \emph{witness} of $P$ having a constant saturating function. Notice also that if $P$ has a constant saturating function then it necessarily has a witness for this, and so a $P$ has a constant saturating function if and only if it has a witness matrix for this.

Thus, to give a decision algorithm, it would be enough to prove that there exists a function $f(k,l)$ such that for any $k\times l$ matrix $P$ if there exists some witness $M$ for $P$ then there is also an  $f(k,l)\times f(k,l)$ matrix $M'$ which is a witness for $P$. Indeed, then it would be enough to check all $f(k,l)\times f(k,l)$ $0$-$1$ matrices, which is $2^{\left(f(k,l)^2\right)}$ matrices. If any of them is a witness matrix, then $P$ has a constant saturating function. If this is not the case, then $P$ has a linear saturating function.


\bibliographystyle{plainurl}
\bibliography{saturation}
\end{document}